\documentclass[12pt]{amsart}

\usepackage{amsmath, amsthm, amssymb}
\usepackage{tikz}
\usepackage{graphicx}
\usepackage{amsfonts}
\usepackage{pgfplots}
\usepackage{comment}
\usepackage{esint}
\usepackage{cancel}
\usepackage{eucal}
\usepackage{enumerate}
\usepackage{xcolor}
\usepackage[colorlinks=true, pdfstartview=FitV, linkcolor=blue, citecolor=blue, urlcolor=blue]{hyperref}

\usepackage{appendix}

\allowdisplaybreaks[4]
\numberwithin{equation}{section}

\newcommand\myworries[1]{\textcolor{red}{#1}}
\renewcommand\myworries[1]{}

\newtheorem{thm}{Theorem}[section]
\newcommand{\bt}{\begin{thm}}
\newcommand{\et}{\end{thm}}

\newtheorem{conj}[thm]{Conjecture}

\newtheorem{cor}[thm]{Corollary}   
\newcommand{\bc}{\begin{cor}}
\newcommand{\ec}{\end{cor}}

\newtheorem{lem}[thm]{Lemma}   
\newcommand{\bl}{\begin{lem}}
\newcommand{\el}{\end{lem}}

\newtheorem{prop}[thm]{Proposition}
\newcommand{\bp}{\begin{prop}}
\newcommand{\ep}{\end{prop}}

\newtheorem{defn}[thm]{Definition}
\newcommand{\bd}{\begin{defn}}    
\newcommand{\ed}{\end{defn}}

\newtheorem{rmrk}[thm]{Remark}   

\newcommand{\br}{\begin{rmrk}}
\newcommand{\er}{\end{rmrk}}

\newcommand{\thmref}[1]{Theorem~\ref{#1}}

\newcommand{\lemref}[1]{Lemma~\ref{#1}}
\newcommand{\defref}[1]{Definition~\ref{#1}}

\newcommand{\be}{\begin{equation}}
\newcommand{\ee}{\end{equation}}

\newcommand{\N}{\mathbb{N}}

\newcommand{\R}{\mathbb{R}}

\newcommand{\g}{\overline{g}}

\newcommand{\on}{\overline{\nabla}}

\newcommand{\vol}{{\rm vol}}
\newcommand{\Sc}{{\rm Sc}}
\newcommand{\Ric}{{\operatorname{Ric}}}

\begin{document}

\title[PSC with conical singularity]{Positive scalar curvature and isolated conical singularity}

\author{Xianzhe Dai}
\address{
Department of Mathematics,
University of California, Santa Barbara
CA93106, USA}
\email{dai@math.ucsb.edu}

\author{Yukai Sun}
\address{Key Laboratory of Pure and Applied Mathematics,
School of Mathematical Sciences, Peking University, Beijing, 100871, P. R. China
}
\email{sunyukai@math.pku.edu.cn}

\author{Changliang Wang}
\address{
School of Mathematical Sciences and Institute for Advanced Study, Key Laboratory of Intelligent Computing and Applications(Ministry of Education), Tongji University, Shanghai 200092, China}
\email{wangchl@tongji.edu.cn}

\date{}


\begin{abstract}
We prove a Geroch type result for isolated conical singularity. Namely, we show that there is no Riemannian metric $g$ on $ X \# T^n $ with an isolated conical singularity which has nonnegative scalar curvature on the regular part, and is positive at some point. In particular, this implies that there is no metric on tori with an isolated conical singularity and positive scalar curvature. We also prove that a scalar flat Riemannian metric $g$ on $X \# T^n$ with finitely many isolated conical singularities must be flat, and extend smoothly across the singular points. We do not a priori assume that a conically singular point on $X$ is a manifold point; i.e., the cross section of the conical singularity may not be spherical.
\end{abstract}

\maketitle

\tableofcontents

\section{Introduction}

One of the fundamental results in the study of scalar curvature is that the $n$-torus $T^n$ does not admit any complete Riemannian metric of positive scalar curvature. This result, was known as the Geroch conjecture, was first proved by Schoen-Yau \cite{SY-MM-79} in dimensions $ n \leq 7 $ using the minimal surface method, and later by Gromov-Lawson \cite{GL-Annals-80, GL-IHESP-83} in arbitrary dimensions using the Dirac operator method. More generally, Gromov-Lawson\cite{GL-IHESP-83} proved that a $\Lambda^2$-enlargeable manifold does not admit any complete metric with positive scalar curvature. For the notion of $\Lambda^2$-enlargeability, we refer to \cite{GL-IHESP-83}, also \cite{Wang-Zhang-22} in which this notion is referred as area enlargeability. We note that tori $T^n$ are $\Lambda^2$-enlargeable. If $W$ is a closed and $\Lambda^2$-enlargeable manifold, then for any closed spin manifold $M$ of the same dimension, the connected sum $M \# W$ is also $\Lambda^2$-enlargeable, in particular, so is $M \# T^n$. In \cite{Wang-Zhang-22}, Wang-Zhang extended Gromov-Lawson's result to allow the spin manifold $M$ to be non-compact. They proved that if $W$ is a closed $\Lambda^2$-enlargeable manifold and $M$ is an arbitrary spin manifold of the same dimension, then the connected sum $M \# W$ does not admit any complete metric of positive scalar curvature. In the non-spin setting, Chodosh-Li \cite{Chodosh-Li-Annals} proved that for $3 \leq n \leq 7$ and any manifold $M$ of dimension $n$, there is no complete metric of positive scalar curvature on $M \# T^n$. In dimension three, this was also proved by Lesourd-Unger-Yau \cite{LUY-JDG}. Recently, Chen-Chu-Zhu \cite{CCZ2023} proved that for a $n$-dimensional aspherical manifold $N$ and any $n$-dimensional manifold $M$, the connected sum $M\# N$ does not admit a complete positive scalar curvature metric for $3\leq n\leq 5$.

The resolution of the Geroch Conjecture and its generalizations have played crucial roles in many important studies concerning scalar curvature, for example Schoen-Yau's resolution of the positive mass theorem \cite{SY-PMT} and Schoen's resolution of the Yamabe problem \cite{Schoen-Yamabe-problem}. Recently, Zhu \cite{Zhu-IRMN-23} has proved that the generalized Geroch conjecture implies the positive mass theorem with arbitrary ends, which has also been proved by Lesourd-Unger-Yau \cite{LUY-JDG}.

Motivated by the investigation of weak notions of nonnegative scalar curvature, e.g., Gromov's polyhedral comparison theory \cite{Gromov14}, and the positive mass theorem for singular metrics, it is natural to study the existence problem of incomplete or singular Riemannian metrics with positive scalar curvature. In this paper, we focus on metrics with isolated conical singularities, and our main result is as follows.
\begin{thm}\label{thm-main-1}
Let $M$ be an $n$-dimensional closed smooth manifold with $n\geq 3$. Assume that either 
$n\leq 7$, or $M$ is spin. 
Then there is no Riemannian metric $g$ on $M \# T^n$ with finitely many isolated conical singularities, 
 and positive scalar curvature on the regular part. Moreover, the conically singular metric with nonnegative scalar curvature must be flat everywhere on $M \# T^n$, and extend smoothly across the singular points.
\end{thm}

Here positive scalar curvature means that the scalar curvature is nonnegative and strictly positive somewhere.

\begin{rmrk}\label{rmrk-thm-1}
{\rm
In Theorem \ref{thm-main-1}, when $M$ is spin, the non-existence conclusion still holds if we replace $T^n$ by any $\Lambda^2$-enlargeable closed manifold of the same dimension, by the work of Wang-Zhang in \cite{Wang-Zhang-22}.
}
\end{rmrk}

An isolated conical singularity is a point that has a neighborhood diffeomorphic to $(0, 1) \times N$, for some closed manifold $N$, such that the restriction of the metric on this neighborhood is asymptotic to the model cone metric:
\begin{equation}
\g := dr^2 + r^2 g^N,
\end{equation}
where $r$ is a coordinate on $(0, 1)$ and $g^N$ is a smooth Riemannian metric on $N$, for a more precise definition, see Definition \ref{defn-conic-mfld} below. Note that $r=0$ corresponds to the conical singularity. In Theorem \ref{thm-main-1}, the conically singular points are smooth manifold points, i.e. the cross sections are diffeomorphic to the sphere, but the metrics on the cross sections may not be the standard one. The proof of Theorem \ref{thm-main-1} works in more general case. Namely, both the non-existence and rigidity conclusions still hold in the case when the cross section $(N, g^N)$ is allowed to be arbitrary closed Riemannian manifold, and which we state the result in Theorem \ref{thm-mfld-with-bdry} below. Note that even though we do not assume that the singular points are manifold points initially, in the rigidity part, 
they are forced to be manifold points, and in particular the cross sections are standard spheres.

As a direct consequence of Theorem \ref{thm-main-1} and Remark \ref{rmrk-thm-1}, we have the following corollary:
\begin{cor}\label{cor-torus}
There is no Riemannian metric on a $\Lambda^2$-enlargeable  closed manifold that has finitely many isolated conical singularities and positive scalar curvature.
\end{cor}

In \cite{CAW-TRAN-85}, Chou proved Corollary \ref{cor-torus} in even dimensions.
In \cite{Shi-Tam-PJM-2018}, Shi-Tam proved Corollary \ref{cor-torus} on tori for metrics with a special class of conical singularity, i.e, the cross section of the model cone in (\ref{eqn-model-cone-metric}) is a scaled round sphere.

The basic idea of proving the non-existence result in Theorem \ref{thm-main-1} is conformally blow-up the singular points to obtain a complete metric with positive scalar curvature. This would contradict the results in \cite{Chodosh-Li-Annals} and \cite{Wang-Zhang-22}. We provide two proofs here. 
In the first proof, we first transplant the problem to an asymptotic flat manifold by doing a connect sum with a well-chosen asymptotic flat manifold. Then the conformal factor is given by a Green's type function of the Laplace operator, which is solved on the asymptotically flat manifold with a conical singularity. This proof is given in Section \ref{sect-case-1}. In the second proof, we work with the original manifold and choose the conformal factor given by a Green's type function of the conformal Laplacian on the compact manifold with a conical singularity. The second approach is similar to the one used for uniformly Euclidean singular metrics in \cite{Li-Mantoulidis, Cecchini2024, WX2024}. However, our conical singularity also allows for a topological singularity, which would not be uniformly Euclidean. The analysis is provided in Section \ref{sect-analysis}, and the alternative proof of the non-existence result in Theorem \ref{thm-main-1} is given in Section \ref{sect-non-existence}.

To prove the rigidity result in Theorem \ref{thm-main-1}, we first note that a nonnegative scalar curved Riemannian metric with a conical singularity must be scalar flat, by the non-existence result in Theorem \ref{thm-main-1}. Then, one shows that the conically singular metric must be Ricci flat by contradiction. Otherwise, we could locally deform the metric along the Ricci curvature direction and then do a conformal change for the deformed metric by using a first eigenfunction of the conformal Laplacian of the deformed metric. This would yield a metric with positive scalar curvature and a conical singularity, which contradicts the nonexistence result in Theorem \ref{thm-main-1}. This idea follows from a similar argument for smooth compact manifolds in \cite{Kazdan-85}. In our singular setting, we need to employ the spectral property of the conformal Laplacian on compact manifolds with a isolated conical singularity, as established in \cite{DW-JGA-2018}. (On compact smooth manifolds, one can also globally deform metric using Ricci flow to increase scalar curvature; however, the study of Ricci flow on conically singular manifolds becomes much more subtle.) Once the metric is shown to be Ricci flat away from the conical singularity, we can apply the fundamental group rigidity result for RCD space proved by Mondino-Wei in \cite{Mondino-Wei-19} to complete the proof. This part of the argument will be done in Section \ref{sect-rigidity}.

Kazdan-Warner \cite{KW-JDG-75} and Schoen \cite{Schoen-89} proved that on a closed manifold, there exists a smooth metric with positive scalar curvature iff its Yamabe constant (aka $\sigma$-constant or Schoen constant) is positive. Moreover, their rigidity result says that on a closed manifold with a nonpositive Yamabe constant, any metric with nonnegative scalar curvature must be Ricci flat. Motivated by the study of weak notions of positive scalar curvature, Li-Mantoulidis raised the question (Question 1.2 in \cite{Li-Mantoulidis}) of extending the previous rigidity result to the singular setting. Namely, for a closed manifold with a nonpositive Yamabe constant, they asked what conditions on a continuous, uniformly Euclidean singular metric with nonnegative scalar curvature can ensure that the metric can be extended smoothly to the whole manifold and is Ricci flat. Here uniformly Euclidean metrics are those that are quasi-isometric to smooth metrics. In \cite{Li-Mantoulidis}, Li-Mantoulidis answered affirmatively this question for codimension $2$ conically singular metrics (aka edge metric) with cone angles $\leq 2\pi$. They also constructed a codimension $2$ edge metric with cone angles  $>2\pi$ as a counterexample. Therefore, for codimension $2$ singular metrics, additional geometric assumptions are needed to ensure an affirmative answer to the question. In contrast, in the case of higher codimensions, Schoen conjectured that no extra regularity assumption is needed:
\begin{conj}[ Conjecture 1.5 in \cite{Li-Mantoulidis}]\label{conj-Schoen}
{\rm
Suppose $M$ is a closed smooth manifold whose Yamabe constant is nonpositive, and $g$ is a uniformly Euclidean metric on $M$ that is smooth away from a closed, embedded submanifold $S \subset M$ with codimension $(S \subset M) \geq 3$. Then
\begin{eqnarray*}
& & \Sc_g \geq 0 \ \  \text{on} \ \ M \setminus S \ \ \text{and} \ \ \sigma(M) \leq 0  \\
& \Longrightarrow & g \ \ \text{extends smoothly to } \ \ M \ \ \text{and} \ \ \Ric_g \equiv 0.
\end{eqnarray*}
}
\end{conj}
Li-Mantoulidis \cite{Li-Mantoulidis} confirmed Conjecture \ref{conj-Schoen} in dimension three. Kazaras \cite{Kazaras-MAAN-19} proved that, on a four-dimensional enlargeable closed smooth manifold, a uniformly Euclidean metric $g$  that is smooth away from a codimension three submanifold and has positive scalar curvature on the regular part must be Ricci flat on the regular part. Cecchini-Frenck-Zeidler \cite{Cecchini2024} proved that a uniformly Euclidean metric on a closed spin manifold with the fundamental group satisfying some condition on Rosenberg index must be Ricci flat on the regular part, provided that the metric is smooth away from a finite set and has positive scalar curvature on the regular part. In particular, their work applies to manifolds of form $M \# T^n$, where $M$ is an arbitrary spin manifold of dimension $n$. Wang-Xie \cite{WX2024} proved that on a closed spin enlargeable $n$-dimensional smooth manifold, a uniformly Euclidean metric $g$ that is smooth away from an embedded finite simplicial complex with codimension$\geq \frac{n}{2} +1$ and has positive scalar curvature on the regular part must be Ricci flat on the regular part, provided the fundamental group of the singular set maps to zero in the fundamental group of the whole manifold.

In contrast to the positive results mentioned above, Cecchini-Frenck-Zeidler \cite{Cecchini2024} also constructed some counterexamples to Conjecture \ref{conj-Schoen} on simply connected closed smooth manifold of dimension$\geq 8$. However, if the manifold has certain topology, then Conjecture \ref{conj-Schoen} could still be true. Indeed, Theorem \ref{thm-main-1} confirms Conjecture \ref{conj-Schoen} in the case of metrics with isolated conical singularities on $M \# T^n$, where $M$ is a closed spin manifold of arbitrary dimension $n$, or a closed manifold of dimension $n \leq 7$. Note that in the case of isolated  conical singularities, by carefully using the geometry and analysis on a cone, we are able to obtain the removable singularity result; that is, the metric can be smoothly extended across the singular points. In the works about general uniformly Euclidean singular metrics mentioned above, such smooth extension results have not yet been obtained yet. It is an interesting question to be investigated.

For a conically singular metric on a smooth manifold, one can also use the work in \cite{LSW63} and \cite{SY-MM-79} to solve a Green's type function to conformally blow-up the singular metric, similar to the approach used for general uniformly Euclidean singular metric in \cite{Li-Mantoulidis, Kazaras-MAAN-19, Cecchini2024, WX2024}. However, this method cannot deal with a conical singularity that is a topological singularity, that is, the cross section is not diffeomorphic to a sphere. This can be viewed as conically pinching the boundary to a point for a manifold with boundary (if connected, otherwise, multiple points) as discussed in the following. Therefore, we apply analysis to asymptotically flat and closed manifolds with isolated conical singularities to solve for the desired conformal factor.

For a compact manifold $M$ with a non-empty boundary $\partial M$, on pp. 31-32 in \cite{Gromov-four-lect-v3}, Gromov raised a Fill-in problem, asking whether any given metric on the boundary $\partial M$ can be extended to a metric on $M$ with positive scalar curvature. In \cite{Shi-Wang-Wei-Crelle-22}, Shi-Wang-Wei answered affirmatively this question of Gromov. On the other hand, under the assumption that the boundary is a sphere $\mathbb{S}^{n-1}$ ($3 \leq n \leq 7$) with a metric satisfying some geometric bound, Shi-Wang-Wei-Zhu \cite{SWWZ-MAAN-21} proved that if, in addition, the mean curvature of the boundary is prescribed and too large, then there is no metric extension with nonnegative scalar curvature. Interestingly, note that if we remove the singular points of the conically singular Riemannian manifold considered in Theorem \ref{thm-main-1}, then the remaining regular part is diffeomorphic to a manifold with boundary whose components are spheres. Also note that the mean curvature of the cross sections of a cone tends to positive infinity at  the cone tip. Different from the setting that Shi-Wang-Wei-Zhu worked on, we pinch geometrically the boundary to a point (see Figure 1). Actually, the proof of Theorem \ref{thm-main-1} works for a more general setting where the cross section could be any closed manifold. So for manifolds with boundary, we have the following result.
\begin{thm}\label{thm-mfld-with-bdry}
Let $M$ be an $n$-dimensional compact manifold with boundary, $n\geq 3$. Assume either that the dimension $n \leq 7$, or that $M$ is a spin manifold. Then there is no Riemannian metric on $M \# T^n$ with a conical singularity and positive scalar curvature away from the singularity and pinching each component of the boundary into a conically singular point. Moreover, such a conically singular metric with nonnegative scalar curvature must be flat, and in particular, each component of the boundary must be diffeomorphic to a sphere. 
\end{thm}
\begin{center}
\begin{tikzpicture}
\draw (-2.3,-.25) arc (-150:150:1 and 0.5);
\draw  (-2.0,0) arc (180:360: 0.5 and .2);
\draw  (-1.2,-.15) arc (0:180: 0.3 and .15);
\draw  (-2.646,0.15) arc (270:330: 0.4 and .2);
\draw  (-2.3,-0.25) arc (30:90: 0.4 and .2);
\draw  (-2.646,0.15) arc (90:120: 0.4 and .2);
\draw  (-2.846,-0.123) arc (240:270: 0.4 and .2);
\draw (-3.23,0)--(-2.83,0.128);
\draw (-3.23,0)--(-2.83,-0.128);
\draw (-2.9,0.1) arc (90:270: 0.04 and 0.1);
\draw [densely dashed] (-2.9,-0.1) arc (-90:90: 0.04 and 0.1);
\draw (-8.8,-.25) arc (-150:150:1 and 0.5);
\draw  (-8.5,0) arc (180:360: 0.5 and .2);
\draw  (-7.7,-0.15) arc (0:180: 0.3 and .15);
\draw  (-9.146,0.15) arc (270:330: 0.4 and .2);
\draw  (-8.8,-0.25) arc (30:90: 0.4 and .2);
\draw (-9.6,0.15)--(-9.15,0.15);
\draw (-9.6,-0.15)--(-9.15,-0.15);
\draw (-9.6,0.15) arc (90:270: 0.08 and 0.15);
\draw [densely dashed] (-9.6,-0.15) arc (-90:90: 0.08 and 0.15);

\draw [->] (-6.3,0)--(-3.8,0);
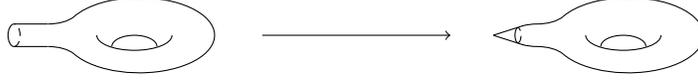
\node at (-5,-1){Figure 1. Conically pinching the boundary};

\end{tikzpicture}
\end{center}

An important application of the resolution of the Geroch conjecture and its generalizations is the positive mass theorem. For smooth asymptotically flat manifolds, a well-known Schoen-Yau-Lohkamp compactification argument reduces the positive mass theorem to the Geroch conjecture for $M \# T^n$, where $M$ is a compact smooth manifold. Kazaras \cite{Kazaras-MAAN-19} and Cecchini-Frenck-Zeidler \cite{Cecchini2024} obtained the positive mass theorem on certain uniformly Euclidean asymptotically flat manifolds by compactification argument. For an asymptotically flat manifold with isolated conical singularities studied in \cite{DSW-spin-23, DSW-PMT-nonspin}, by using the analysis near conical singularity in \cite{DSW-PMT-nonspin} and the analysis on asymptotically flat ends in \cite{SY-PMT}, one can check that Schoen-Yau-Lohkamp compactification argument in \cite{Lohkamp1999} still works, and the positive mass theorem for asymptotically flat manifolds with isolated conical singularities follows from the Geroch conjecture with isolated conical singularities for $M \# T^n$ as established in Theorems \ref{thm-main-1} and \ref{thm-mfld-with-bdry}. 

 
The paper is organized as follows: In section \ref{sect-case-1}, we introduce the basic definition and prove the non-existence result in Theorem \ref{thm-main-1}. In section \ref{sect-analysis}, we prepare some analysis results for section \ref{sect-non-existence}. In section \ref{sect-non-existence}, we provide another way to prove the non-existence result in Theorem \ref{thm-main-1}. In section \ref{sect-rigidity}, we deal with the rigid case. In Appendix \ref{sect-compactification}, we provide some details of the derivation for the Schoen-Yau-Lohkamp compactification with isolated conical singularities.

{\em Acknowledgement}:
 Xianzhe Dai is partially supported by the Simons Foundation. Yukai Sun is partially funded by the National Key R\&D Program of China Grant 2020YFA0712800. Changliang Wang is partially supported by the Fundamental Research Funds for the Central Universities and Shanghai Pilot Program for Basic Research.



\section{Proof of non-existence result in Theorem \ref{thm-main-1}}\label{sect-case-1}


We first give the definition of manifolds with finitely many isolated conical singularities; for simplicity, we give the defintion for 
a single conical singularity.
\begin{defn}\label{defn-conic-mfld}
{\rm
We say $(M^n, g, d, o)$ is a compact Riemannian manifold with a single conical singularity at $o$, if
\begin{enumerate}[(\romannumeral1)]
\item $d$ is a metric on $M$ and $(M, d)$ is a compact metric space, \myworries{XD: I am worried about what does it mean by metric space with smooth boundary}
\item $g$ is a smooth Riemannian metric on the regular part $ M \setminus \{ o \}$,
         $d$ is the induced metric by the Riemannian metric $g$ on $M \setminus \{ o \}$ ,
\item there exists a neighborhood $U_o$ of $o$ in $M \setminus \partial M$, such that
         $U_o \setminus \{ o \} \simeq (0, 1) \times N$
          for a smooth compact manifold $N$,
           and on $U_o \setminus \{ o \}$ the metric $g = \g + h $, where
         \begin{equation}\label{eqn-model-cone-metric}
         \g = dr^2 + r^2 g^N,
         \end{equation}
         $g^N$ is a smooth Riemannian metric on $N$, $r$ is a coordinate on $(0, 1)$, $r=0$ corresponding the singular point $o$,
         and $h$ satisfies
         \begin{equation}\label{eqn-cone-metric-asymptotic}
         |\on^k h|_{\g} = O(r^{ \alpha - k }),  \ \ \text{as} \ \ r \rightarrow 0,
         \end{equation}
         for some $\alpha >0$ and $k = 0, 1 $ and $2$, where $\on$ is the Levi-Civita connection of $\g$.
\end{enumerate}

}
\end{defn}
We will need the notion of asymptotically flat manifolds with finitely many isolated conical singularities, which we recall below.
\begin{defn}\label{defn-AF-conic-mfld}
{\rm
We say $(M^n, g, o)$ is an asymptotically flat manifold with a single isolated conical singularity at $o$, if $M^n = M_0 \cup M_\infty$ satisfies
\begin{enumerate}[(\romannumeral1)]
\item $( M_0, g|_{M_0 \setminus\{o\}}, o)$ is a compact Riemannian manifold with smooth boundary and a single conical singularity at $o$
         defined as in Definition \ref{defn-conic-mfld},
\item  $M_\infty$ is diffeomorphic to  $ \R^n \setminus B_{R}(0) $ for some $R > 0$, and under this diffeomorphism the smooth Riemannian metric $g$ on $M_\infty$ satisfies
         \[
          g = g_{\R^n} + O(\rho^{-\tau}), \ \ |(\nabla^{g_{\R^n}})^i g|_{g_{\R^n}} = O(\rho^{-\tau - i}),  \ \  \text{as} \ \ \rho \rightarrow +\infty,
         \]
         for $i = 1, 2$ and $3$,
         where $\tau > 0$ is the asymptotic order, $\nabla^{g_{\R^n}}$ is the Levi-Civita connection of the Euclidean metric $g_{\R^n}$, and $\rho$ is the Euclidean
         distance to a base point.
\end{enumerate}

}
\end{defn}
For our purpose, we need to produce an asymptotically flat manifold with a conical singularity from a compact manifold with a single conical singularity. We start with:
\begin{lem} \label{lem-AF-PSC}
The manifold:
\[\hat{M}^{n}\cong \mathbb{R}^{n}\# (\mathbb{R}P^3 \times \mathbb{S}^{n-3} )\]
admits a complete asymptotically flat Riemannian metric $g_{\hat{M}^n}$ of asymptotic order $n-2$ with positive scalar curvature.
\end{lem}

\begin{proof}
The proof follows that for similar conclusion in dimension three in Corollary 3.2 in \cite{HR2012}. We take $\mathbb{S}^n$ with the standard metric and $\mathbb{R}P^3 \times \mathbb{S}^{n-3}$ with the standard product metric. By doing surgery in \cite{Gromov-Lawson-80} and \cite{SY-MM-79}, we then obtain a metric $g$ on $\mathbb{S}^n \# (\mathbb{R}P^3 \times \mathbb{S}^{n-3})$ with positive scalar curvature. Let $P$ be a point in $\mathbb{S}^n \subset \mathbb{S}^n \# (\mathbb{R}P^3 \times \mathbb{S}^{n-3})$, and $G_P$ be the Green function of the conformal Laplacian $-\Delta_g + \frac{(n-2)}{4(n-1)}\Sc_g$ with a pole at $P$ on $\mathbb{S}^n \# (\mathbb{R}P^3 \times \mathbb{S}^{n-3})$. Define the metric $\hat{g} =: \left( (n-2) \omega_n G_P \right)^{\frac{4}{n-2}} g$ on $\left(\mathbb{S}^n \setminus \{P\} \right) \# (\mathbb{R}P^3 \times \mathbb{S}^{n-3})$. Note that $\left(\mathbb{S}^n \setminus \{P\} \right) \# (\mathbb{R}P^3 \times \mathbb{S}^{n-3})$ is diffeomorphic to $\mathbb{R}^n \# (\mathbb{R}P^3 \times \mathbb{S}^{n-3})$, and $\hat{g}$ is asymptotically flat and scalar flat, but not Ricci flat. The asymptotic order of $g$ is $n-2$, since $g$ is conformally flat in a neighborhood of $P$, see e.g. Lemma 6.4 in \cite{Lee-Parker}. Finally, by deforming the metric $\hat{g}$ along Ricci flow, we obtain an asmyptotically flat metric $g_{\hat{M}^n}$ of order $n-2$ with positive scalar curvature, see \cite{DM2007} and also \cite{Li-2018}.
\end{proof}

Let $(M \# T^n, g, o), n\geq 3$, be an $n$-dimensional compact manifold with a single conical singularity at $o$. Assume that the scalar curvature of $g$ on the regular part is nonnegative, and is strictly positive at a point $x$. Then by Theorem 1.1 in \cite{Gromov-Lawson-80}, we make a connected sum of $(M \# T^n, g)$ (in a small neighborhood of $x$) and $(\hat{M}^n, g_{\hat{M}^3})$, which is given in \lemref{lem-AF-PSC}, to obtain an asymptotically flat manifold with a single conical singularity at $o$, denoted by $(\widetilde{M} \# T^n, \widetilde{g}, o)$, such that $\widetilde{g}$ has nonnegative scalar curvature on the regular part and strictly positive at some point. 

We recall an existence result for harmonic functions with certain asymptotic behavior on asymptotically flat manifolds with isolated conical singularity proved in \cite{DSW-PMT-nonspin}.
\begin{lem}[Lemma 4.1 in \cite{DSW-PMT-nonspin}]\label{lem-harmonic-function-ACyl}
There is a harmonic function $u \geq 1$ on $(\widetilde{M} \# T^n, \widetilde{g})$ satisfying the following asymptotic conditions near conical singularity and the infinity on the asymptotic flat end:
\begin{equation}\label{eqn-asymptotic-harmonic-function-ACyl}
      u=\begin{cases}
      r^{2-n} + o(r^{2-n+\alpha^\prime}),  & \text{as} \ \ r \to 0, \\
      
        1 + A\rho^{2-n} + o(\rho^{\beta'}), & \text{as} \ \ \rho \to +\infty,\\
      
      \end{cases}
\end{equation}
where $0 < \alpha^\prime < \alpha$ and $\beta' < 2-n$, $A$ is a positive constant. Here $\alpha$ is the asymptotic order of the metric $g$ near the conical singularity in Definition \ref{defn-conic-mfld}.
\end{lem}

Now we are ready to prove the following non-existence result.

\begin{proof}[Proof of non-existence result in Theorem \ref{thm-main-1}]
We prove it by contraction. Assume that there exists a Riemannian metric $g$ on $M \# T^n$ with a single conical singularity at $o$, whose scalar curvature $\Sc_g$ on the regular part is nonnegative and strictly positive somewhere. Then, as described right after Lemma \ref{lem-AF-PSC}, by applying the surgery of Gromov and Lawson \cite{Gromov-Lawson-80}, we obtain an asymptotically flat manifold with a single conical singularity at $o$, denoted by $(\widetilde{M} \# T^n, \widetilde{g}, o)$, whose scalar curvature $\Sc_{\widetilde{g}} \geq 0$ and strictly positive at some point.

Now we use the harmonic function $u$ in \lemref{lem-harmonic-function-ACyl}  to do a conformal change for the metric $\widetilde{g}$, and let
\begin{equation}
\widetilde{g}_{u} := u^{\frac{4}{n-2}} \widetilde{g} \ \ \text{on} \ \ \widetilde{M} \# T^n.
\end{equation}
By the partial asymptotic expansion of $u$ at infinity of the Euclidean end (i.e. as $\rho \to +\infty$) in (\ref{eqn-asymptotic-harmonic-function-ACyl}), $\widetilde{g}_u$ still has an asymptotically flat end. On the other hand, because the surgery is operated locally near $x$, we have that, in a sufficiently small neighborhood of the conically singular point $o$,
\begin{equation}
\widetilde{g} = g = dr^2 + r^2 g^N + h,
\end{equation}
where $h$ satisfies asymptotical control in (\ref{eqn-cone-metric-asymptotic}), and then
\begin{equation}
\widetilde{g}_u
= u^{\frac{4}{n-2}} \widetilde{g}
=  r^{-4}\left( 1+ o(r^{\alpha^\prime})\right) \left(dr^2 + r^2 g^N + h \right), \ \ \text{as} \ \ r \to 0.
\end{equation}
Now we do a change of variable, by letting $s = \frac{1}{r}$, and obtain
\begin{equation}
\widetilde{g}_u = (1 + o(s^{-\alpha^\prime}))(ds^2 + s^2 g^N + \tilde{h}),
\end{equation}
where $|\tilde{h}|_{ds^2 + s^2 g^N} = O(s^{-\alpha})$ as $s \to \infty$. Therefore,
\begin{equation}
\widetilde{g}_u = ds^2 + s^2 g^N + o(s^{-\alpha^\prime}), \ \ \text{as} \ \ s \to +\infty,
\end{equation}
since $0 < \alpha^\prime < \alpha$. Note that the metric $\widetilde{g}_u$ is asymptotically conical, tending to the infinite end of a cone as $s \to \infty$ (i.e. $r \to 0$). In particular, $\widetilde{g}_u$ is a complete Riemannian metric on $\widetilde{M} \# T^n$.

Moreover, the scalar curvature of $\widetilde{g}_u$ is given by
\begin{equation}
\Sc_{\widetilde{g}_u}
= \tfrac{4(n-1)}{n-2}  u^{-\tfrac{n+2}{n-2}} \left( -\Delta_{\widetilde{g}} u + \tfrac{n-2}{4(n-1)}\Sc_{\widetilde{g}} u \right)
=  u^{-\tfrac{4}{n-2}} \Sc_{\widetilde{g}} \geq 0.
\end{equation}
Recall that $\Sc_{\widetilde{g}}$ is strictly positive somewhere, and so is $\Sc_{\widetilde{g}_u}$. This leads to a contradiction with  Theorem 3 in \cite{Chodosh-Li-Annals} respectively Theorem 1.1 in \cite{Wang-Zhang-22}. 
\end{proof}


\section{Analysis on compact Riemannian manifolds with a conical singularity}\label{sect-analysis}

In this section, we provide some analysis preliminaries about the conformal Laplacian on a compact manifold with a conical singularity. This will be used to solve conformal Laplace equation and then give another proof for the non-existence result in Theorem \ref{thm-main-1} in Section \ref{sect-non-existence}.

Let $(M^n, g, o)$ be a compact Riemannian manifold with a single conical singularity at $o$ as in \defref{defn-conic-mfld}. In the rest of this section, we assume that the scalar curvature $\Sc_g \geq 0$. In a conical neighborhood of the singular point, 
\begin{equation}
\Sc_g = \frac{1}{r^2}\left( \Sc_{g^N} - (n-1)(n-2) + O(r^\alpha) \right).
\end{equation}
Thus, $\Sc_g \geq 0$ implies that the scalar curvature on the cross section:
\begin{equation}\label{eqn: scalar bound on cross section}
\Sc_{g^N} \geq (n-1)(n-2).
\end{equation}

We choose a cut-off function $ 0 \leq \chi \leq 1$ such that
\begin{equation}
   \chi(x) = \begin{cases}
             1, & d(x, o) < \epsilon, \\
             0, & d(x, o) > 2 \epsilon,
             \end{cases}
\end{equation}
where $\epsilon > 0$ is chosen sufficiently small.
For each $p \geq 1$, $k \in \N$ and $\delta \in \R$, the {\em weighted Sobolev space} $W^{k, p}_{\delta}(M)$ is defined to be the completion of $C^\infty_0(M \setminus \{o\})$ with respect to the {\em weighted Sobolev norm} given by
\begin{equation}
\|u\|^p_{W^{k, p}_{\delta}(M)} := \int_{M} \sum^{k}_{i=0} \left( r^{-p(\delta -i) -n} |\nabla^i u|^p \chi + |\nabla^i u|^p (1 - \chi) \right),
\end{equation}
where $r$ is the radial coordinate on the conical neighborhood $U_o$ in Definition \ref{defn-conic-mfld} and $|\nabla^i u|$ is the norm of $i^{th}$ covariant derivative of $u$ with respect to $g$. We denote $W^{0, p}_\delta$ by $L^p_\delta$.

By using scaling technique, the usual interior elliptic estimates, and the asymptotic control of metric $g$ near conical point in Definition $\ref{defn-conic-mfld}$, we have the following weighted elliptic estimate.
\begin{prop}\label{prop-weight-elliptic-estimate}
If $u \in L^{p}_{\delta}(M)$, and $(-\Delta + \tfrac{n-2}{4(n-1)} \Sc_g ) u \in L^{p}_{\delta-2}(M)$, then
\be
\|u \|_{W^{2, p}_{\delta}(M)} \leq C \left( \| (-\Delta_g + \tfrac{n-2}{4(n-1)}\Sc_g ) u \|_{L^{p}_{\delta-2}(M)} + \|u\|_{L^{p}_{\delta}(M)} \right)
\ee
holds for some constant $C = C(g, n, k)$ independent of $u$.
\end{prop}

Throughout the paper, we always use $\lambda_j$ to denote the eigenvalues of the Schr\"odinger operator $-\Delta_{g^N} + \tfrac{n-2}{4(n-1)}\Sc_{g^N} - \frac{(n-2)^2}{4}$ on the cross section $(N, g^N)$ of the model cone $(C(N), \overline{g})$ in $(\ref{eqn-model-cone-metric})$. By (\ref{eqn: scalar bound on cross section}), clearly $\lambda_j \geq 0$.  We set
\begin{equation}\label{eqn-nu-defn}
\nu_j^{\pm} := \frac{-(n-2)\pm\sqrt{(n-2)^2+4\lambda_{j}}}{2}.
\end{equation}

\begin{defn}\label{defn-critical-cone}
{\rm
We say $\delta \in \R$ is {\em critical } if $\delta = \nu_j^{+}$ or $\delta = \nu_j^{-}$ for some $j \in \mathbb{Z}_{\geq 0}$, where $\nu_j$ is defined as in (\ref{eqn-nu-defn}).
}
\end{defn}

The estimate for Laplacian in Proposition 4.5 in \cite{DSW-PMT-nonspin} can be adapted to the conformal Laplacian, and we have the following:
\begin{prop}\label{prop-refined-weighted-elliptic-estimate}
If $\delta \in \R$ is not critical as in Definition $\ref{defn-critical-cone}$, there exists a constant $C= C(g, n)$ and a compact set $B \subset M \setminus \{o\}$ such that for any $u \in L^2_{\delta}(M)$ with $-\Delta u + \tfrac{(n-2)}{4(n-1)}\Sc_g u \in L^{2}_{\delta-2}(M)$,
\be
\|u\|_{W^{2, 2}_{\delta}(M)} \leq C  
\left( \| (-\Delta_g + \tfrac{n-2}{4(n-1)}\Sc_g ) u \|_{L^{p}_{\delta-2}(M)} + \|u\|_{L^{2}(B)} \right).
\ee
\end{prop}
A direct consequence of Proposition \ref{prop-refined-weighted-elliptic-estimate} is the Fredholm property of the unbounded operator:
\begin{eqnarray*}
( -\Delta_{g} + \tfrac{n-2}{4(n-1)} \Sc_g )_\delta: {\rm Dom} ( -\Delta_{\delta} + \tfrac{n-2}{4(n-1)} \Sc_g )_\delta & \rightarrow & L^2_{\delta-2}(M) \\
 u \hspace*{1in} & \mapsto & - \Delta_g u + \tfrac{n-2}{4(n-1)}\Sc_g u,
\end{eqnarray*}
where ${\rm Dom} ( -\Delta_{g} + \tfrac{n-2}{4(n-1)}\Sc_g )_\delta$ is dense subset of $L^2_{\delta}(M)$ consisting of function $u$ such that $-\Delta_g u + \tfrac{n-2}{4(n-1)}\Sc_g u \in L^2_{\delta-2}(M)$ in the sense of distributions.

\begin{prop}\label{prop-Fredholm}
If $\delta \in \R$ is not critical as in Definition $\ref{defn-critical-cone}$, then the operator $ \big( -\Delta_{g} + \tfrac{n-2}{4(n-1)}\Sc_g \big)_\delta $ is Fredholm, namely, it has closed range and 
\begin{enumerate}[$(1)$]
\item ${\rm dim}( {\rm Ker}( -\Delta_g + \tfrac{n-2}{4(n-1)}\Sc_g )_\delta ) < +\infty$,
\item ${\rm dim} ( {\rm Ker} ( -\Delta_g + \tfrac{n-2}{4(n-1)}\Sc_g )^*_\delta ) < +\infty$.
\end{enumerate}
\end{prop}
For the proof of Proposition \ref{prop-Fredholm}, we refer to the proof of Proposition 4.8 in \cite{DSW-PMT-nonspin}. In the proof, we will be using $( -\Delta_g + \tfrac{n-2}{4(n-1)}\Sc_g )^*_\delta$, the adjoint operator of $( -\Delta_g + \tfrac{n-2}{4(n-1)}\Sc_g )_\delta$. By using the $L^2$ pairing $(\cdot, \cdot)_{L^2(M)}$, it is given as
\begin{eqnarray*}
( -\Delta_g + \tfrac{n-2}{4(n-1)}\Sc_g )^*_\delta : {\rm Dom} ( ( -\Delta_g + \tfrac{n-2}{4(n-1)}\Sc_g )^*_\delta ) & \rightarrow & L^{2}_{-\delta-n} \\
  u \hspace*{1in} & \mapsto & -\Delta_g u + \tfrac{n-2}{4(n-1)}\Sc_g,
\end{eqnarray*}
where the domain ${\rm Dom} ( ( -\Delta_g + \tfrac{n-2}{4(n-1)}\Sc_g )^*_\delta )$ is the dense subset of $L^{2}_{-\delta+2-n}$ consisting of functions $u$ such that $-\Delta_g u + \tfrac{n-2}{4(n-1)}\Sc_g u \in L^2_{-\delta-n}$ in the distributional sense.

Then with the help of the Fredholm property in Proposition \ref{prop-Fredholm}, we prove following surjectivity property.
\begin{prop}
    Let $(M^n, g, o)$ be a compact Riemannian manifold with a conical singularity at $o$ as in Definition \ref{defn-conic-mfld}, and assume that $g$ has positive scalar curvature on the regular part. We have that for any noncritical $\delta \leq \tfrac{2-n}{2}$, the map
\begin{eqnarray*}
( -\Delta_{g} + \tfrac{n-2}{4(n-1)} \Sc_g )_\delta: {\rm Dom} ( -\Delta_{\delta} + \tfrac{n-2}{4(n-1)} \Sc_g ) & \rightarrow & L^2_{\delta-2}(M) \\
 u \hspace*{1in} & \mapsto & - \Delta_g u + \tfrac{n-2}{4(n-1)}\Sc_g u,
\end{eqnarray*}
    is surjective.
\end{prop}
\begin{proof}
Note that $ \tfrac{2-n}{2}$ is not critical as in Definition \ref{defn-critical-cone}, since $\lambda_{j} \geq 0$.

For any noncritical $\delta \leq \tfrac{2-n}{2}$, we have
\be
-\delta + 2 -n \geq \tfrac{2-n}{2}.
\ee
Thus
\be
{\rm Dom}( ( -\Delta_{g} + \tfrac{n-2}{4(n-1)} \Sc_g)^*_\delta )
\subset L^{2}_{-\delta + 2 -n}(M) \subset L^2_{\tfrac{2-n}{2}}(M).
\ee

For every small number $r_0$, we can choose a smooth cut-off function
\be
0 \leq \chi_{r_0} \leq 1
\ee
such that
\be
\chi_{r_0}
= \begin{cases}
    0, & \text{on} \ \ B_{r_0}(o), \\
    1, & \text{on} \ \ ( B_{2r_0}(o))^c ,
   \end{cases}
\ee
and
\be
|d \chi_{r_0}|
\leq
        \tfrac{10}{r_0}.
\ee
For any $u \in {\rm Ker}( ( -\Delta_{g} + \tfrac{n-2}{4(n-1)} \Sc_g )^*_\delta ) \subset L^{2}_{\tfrac{2-n}{2}}(M)$, we have
\begin{eqnarray*}
\int_{M} |  \nabla (\chi_{r_0} u)|^2
& = &  \int_M |d\chi_{r_0}|^2 |u|^2 - \int_{M} \left( \chi_{r_0} \right)^2 \langle u, \Delta_g u \rangle \\
& = & \int_M |d\chi_{r_0}|^2 |u|^2 - \int_{M} \left( \chi_{r_0} \right)^2 \tfrac{n-2}{4(n-1)} \Sc_g u^2 \\
& \leq & \int_M |d\chi_{r_0}|^2 |u|^2,
\end{eqnarray*}
since $ \Sc_g \geq 0$ on $M$.
This then implies
\begin{eqnarray}
\int_{(B_{2r_0}(o))^c} |\nabla \varphi|^2
 \leq   C \left( \int_{A_{r_0}} |u|^2 r^{-2}\right)
 \rightarrow  0, \ \ \text{as} \ \ r_0 \rightarrow 0,
\end{eqnarray}
since
\be
\| u \|^2_{L^2_{\tfrac{2-n}{2}(M)}} =  \int_{M}  \left( r^{-2} | \varphi |^2 \chi + | \varphi |^2 (1-\chi) \right) d\vol_{g} < \infty.
\ee
Here $A_{r_0}$ is the annular region $B_{2r_0}(o) \setminus B_{r_{0}}(o)$.
Thus $\nabla u =0$, hence $u$ must be a constant function. As a result,
\begin{equation}
    0 = -\Delta_g u + \tfrac{n-2}{4(n-1)} \Sc_g u = \tfrac{n-2}{4(n-1)} \Sc_g u,
\end{equation}
and so $u \equiv 0$, since $\Sc_g$ is not identically zero.
This shows  ${\rm Ker}( ( -\Delta_g + \tfrac{n-2}{4(n-1)} \Sc_g )^*_\delta ) = \{0\}$.

Now for any fixed $v \in C^\infty_0 (M)$, if
\be
(v, ( -\Delta_g + \tfrac{n-2}{4(n-1)} \Sc_g )_\delta u )_{L^2(M)} = 0,  \ \ \forall u \in {\rm Dom}( ( -\Delta_g + \tfrac{n-2}{4(n-1)} \Sc_g )_\delta),
\ee
then
\be
( ( -\Delta_g + \tfrac{n-2}{4(n-1)} \Sc_g )^*_\delta v, u )_{L^2(M)} = 0, \ \ \forall u \in {\rm Dom}( ( -\Delta_g + \tfrac{n-2}{4(n-1)} \Sc_g )_\delta ).
\ee
By the density of $C^\infty_0(M) \subset {\rm Dom}( ( -\Delta_g + \tfrac{n-2}{4(n-1)} \Sc_g )_\delta )$ in $L^2(M)$, this implies that $( -\Delta_g + \tfrac{n-2}{4(n-1)} \Sc_g )^*_\delta v = 0$, and so $v = 0$. Therefore, ${\rm Ran}(( -\Delta_g + \tfrac{n-2}{4(n-1)} \Sc_g )_\delta)^{\perp}\cap C^\infty_0(M)=\{0\}$. As a result, the density of $C^\infty_0(M)$ and closeness (follows from Proposition \ref{prop-Fredholm}) of ${\rm Ran}( -\Delta_{\delta} + \tfrac{n-2}{4(n-1)} \Sc_g)$ in $L^2_{\delta-2}(M)$ implies that ${\rm Ran}(( -\Delta_g + \tfrac{n-2}{4(n-1)} \Sc_g )_\delta) = L^2_{\delta-2}(M)$, i.e. $( -\Delta_g + \tfrac{n-2}{4(n-1)} \Sc_g )_\delta$ is surjective.
\end{proof}

By Definition \ref{defn-critical-cone} and the fact that $\lambda_j \geq 0$, one can easily see that there is no critical index in the interval $(2-n, 0)$.
As a result, similarly as in Proposition 4.15 in \cite{DSW-PMT-nonspin}, we can extend the range of noncritical indices $\delta$ for which $-\Delta_\delta + f$ is surjective as following.
\begin{prop}\label{prop-surjectivity}
For noncritical $\delta < 0$ as in Definition \ref{defn-critical-cone},
\be
( -\Delta_g + \tfrac{n-2}{4(n-1)} \Sc_g )_\delta: {\rm Dom} (( -\Delta_g + \tfrac{n-2}{4(n-1)} \Sc_g )_\delta ) \rightarrow L^{2}_{\delta-2}(M)
\ee
is surjective.
\end{prop}

\section{Another proof of the non-existence result in Theorem \ref{thm-main-1}}\label{sect-non-existence}

\begin{lem}\label{lem-conformal-laplacian-solution}
Let $(M^n, g, o)$ be a $n$-dimensional compact Riemannian manifold with a single conical singularity at $o$, and assume that $g$ has positive scalar curvature on the regular part. Then the equation
\begin{equation}
-\Delta_g u + \tfrac{n-2}{4(n-1)} \Sc_g u =0
\end{equation}
has a solution $u$ satisfying the following asymptotic condition near the singular point $o$:
\begin{equation}
u = a(x) r^{\nu^-_1} + o(r^{\nu^-_1+\alpha^\prime}), \ \ \text{as} \ \ r \to 0,
\end{equation}
for $0 < \alpha^\prime < \alpha$, where $\alpha$ is the decay order in (\ref{eqn-cone-metric-asymptotic}), $\nu^-_1 = \tfrac{-(n-2) - \sqrt{(n-1)^2 + 4 \lambda_1}}{2}$, $\lambda_1$ is the least eigenvalue of $-\Delta_{g^N} + \tfrac{n-2}{4(n-1)}\Sc_{g^N} - \tfrac{(n-2)^2}{4}$, and $a(x) >0$ is a corresponding eigenfunction on the cross section $N$.
\end{lem}

\begin{proof}
Choose a cut-off function $\phi$ satisfying
\begin{equation*}
\phi = \begin{cases}
        1, & \text{on} \ \ B_{\frac{1}{2}}(o), \\
        0, & \text{on} \ \ M \setminus B_1(o).
       \end{cases}
\end{equation*}
Let $u_0 = \phi a(x) r^{\nu^-_1}$, which is a smooth function supported in a neighborhood of the conically singular point $o$. Then because $( -\Delta_{\g} + \tfrac{n-2}{4(n-1)} \Sc_{\g} ) a(x) r^{\nu^-_1} = 0$, by the asymptotic control of $g$ near conically singular point $o$ in Definition \ref{defn-conic-mfld}, we have
\begin{equation*}
( -\Delta_g +  \tfrac{n-2}{4(n-1)} \Sc_g )u_0
= O(r^{\nu^-_1 -2 + \alpha}),  \text{as} \ \ r \to 0.
\end{equation*}
By the definition of weighted Sobolev spaces, this implies that
\begin{equation*}
( -\Delta_g +  \tfrac{n-2}{4(n-1)} \Sc_g ) u_0 \in L^2_{\delta-2}(M), \quad \forall \delta < \nu^-_1 + \alpha.
\end{equation*}
Then by applying Proposition \ref{prop-surjectivity}, we obtain $v \in L^{2}_{\delta}$ such that
\begin{equation*}
( -\Delta_g + \tfrac{n-2}{4(n-1)} \Sc_g ) v = ( -\Delta_g + \tfrac{n-2}{4(n-1)} \Sc_g) u_0.
\end{equation*}
So by setting $u = u_0 - v$, we have $( -\Delta_g + \tfrac{n-2}{4(n-1)}\Sc_g ) u = 0$.

Finally, we derive the asymptotic behavior of $v$ near the singular point. Note that for a sufficiently small $\epsilon >0$, in the ball $B_\epsilon (o)$ centered at the singular point $o$, we have $u = a(x)r^{\nu^-_1} - v$, and so
\begin{eqnarray*}
( -\Delta_g + \tfrac{n-2}{4(n-1)} \Sc_g ) v
& = & ( -\Delta_g + \tfrac{n-2}{4(n-1)} \Sc_g ) a(x) r^{\nu^-_1}  \\
& = & O(r^{\nu^-_1 -2 +\alpha}) \in L^{p}_{\delta -2}(B_{\epsilon}(o)), \ \ \forall \delta < \nu^-_1+\alpha \ \ \text{and} \ \  \forall p >1.
\end{eqnarray*}
Then the weighted elliptic estimate in Proposition \ref{prop-weight-elliptic-estimate} implies that $v \in W^{2, 2}_{\delta}(B_{\epsilon}(o))$ for  $\delta < 2-n+\alpha$. Consequently, the weighted Sobolev inequality (e.g. in Proposition 3.4 in \cite{DW-MRL-2020}) implies that $v \in L^{\frac{2n}{n-2}}_{\delta}(B_{\epsilon}(o))$, and by applying weighed elliptic estimate in Proposition \ref{prop-weight-elliptic-estimate} (with $p = \frac{2n}{n-2}$) again, we obtain $v \in W_\delta^{2, \frac{2n}{n-2}}(B_{\epsilon}(o))$ for $ \delta < 2-n+\alpha$. By repeating this process, we can obtain that $v \in W^{2, p}_{\delta}(B_{\epsilon}(o))$ for all $p>1$ and $\delta<2-n+\alpha$, and so by Proposition 3.4 in \cite{DW-MRL-2020}, we have
\begin{equation}
v = o(r^{2-n+\alpha^\prime}), \ \ \text{as} \ \ r \to 0,
\end{equation}
for $0< \alpha^\prime < \alpha$. This completes the proof.
\end{proof}

{\em Another proof of the non-existence result in Theorem \ref{thm-main-1}}: By applying Lemma \ref{lem-conformal-laplacian-solution}, we obtain a function $u$ satisfying the equation
\begin{equation}
-\Delta_g u + \tfrac{n-2}{4(n-1)} \Sc_g u =0,
\end{equation}
and the asymptotic control
\begin{equation}\label{eqn-conformal-laplacian-solution-asymptotic}
u = a(x)r^{\nu^-_1} + o(r^{\nu^-_1+\alpha^\prime}), \ \ \text{as} \ \ r \to 0,
\end{equation}
for $0 < \alpha^\prime < \alpha$. Note that
\begin{equation}
\nu^-_1 = \frac{-(n-2) - \sqrt{(n-1)^2 + 4 \lambda_1}}{2} \leq 2-n,
\end{equation}
since $\lambda_1 \geq 0$, which follow from $\Sc_g \geq 0$.
As a result, the function $u$ tends to $+\infty$ as one approaches the singular point (i.e. $r \to 0$), since $a(x) >0$. As a result,  the infimum of the function $u$ on the whole manifold $M \# T^n$ is achieved in a compact subset away from the singular point, and so $\inf\limits_{M \# T^n \setminus \{o\}} u > -\infty$. Thus we can take a sufficiently large positive number $C$ such that
\begin{equation}
\widetilde{u} :=u + C \geq 1 \ \ \text{on} \ \ M\#T^n \setminus \{o\}.
\end{equation}
Then $\widetilde{u}$ satisfies
\begin{equation}
- \Delta_g \widetilde{u} + \tfrac{n-2}{4(n-1)} \Sc_g \widetilde{u} = C \tfrac{n-2}{4(n-1)} \Sc_g,
\end{equation}
and the asymptotic control as in (\ref{eqn-conformal-laplacian-solution-asymptotic}).

Now we do a conformal change for the metric $g$, and let
\begin{equation}
    \widetilde{g} := (\widetilde{u})^{\frac{4}{n-2}}g \ \ \text{on} \ \ M \# T^n \setminus \{o\}.
\end{equation}
The scalar curvature of $\widetilde{g}$ is given by
\begin{equation}\label{eqn-conformal-scalar-curvature}
\Sc_{\widetilde{g}} = \tfrac{4(n-1)}{(n-2)} \widetilde{u}^{-\tfrac{n+2}{n-2}} \left( - \Delta_g \widetilde{u} + \tfrac{n-2}{4(n-1)} \Sc_g \widetilde{u} \right)
= C \widetilde{u}^{-\tfrac{n+2}{n-2}} \Sc_g \geq 0.
\end{equation}
And $\Sc_{\widetilde{g}}$ is strictly positive at points where $\Sc_g$ is strictly positive.

Moreover, near the conically singular point $o$, the metric $g$ is given as
\begin{equation}
g = dr^2 + r^2 g^N + h,
\end{equation}
where $h$ satisfies (\ref{eqn-cone-metric-asymptotic}), and then
\begin{equation}
\widetilde{g}
= \left( \widetilde{u} \right)^{\tfrac{4}{n-2}} g
=  a(x)^{\tfrac{4}{n-2}}r^{\tfrac{4\nu^-_1}{n-2}}\left( 1+ o(r^{\alpha^\prime})\right) \left(dr^2 + r^2 g^N + h \right), \ \ \text{as} \ \ r \to 0.
\end{equation}
We do a change of variable, by letting $s = \frac{1}{r}$, $s \in (1, +\infty)$, and then rewrite $\widetilde{g}$ as
\begin{equation}
\widetilde{g} = a(x)^{\tfrac{4}{n-2}} s^{-\tfrac{4\nu^-_1}{n-2} -4}(1 + o(s^{-\alpha^\prime}))(ds^2 + s^2 g^N + \widetilde{h}),
\end{equation}
where $|\widetilde{h}|_{ds^2 + s^2 g^N} = O(s^{-\alpha})$ as $s \to \infty$. In addition, because $\nu^-_1 \leq 2-n$, we have $-\frac{4 \nu^-_1}{n-2} - 4 \geq 0$. Thus, $\widetilde{g}$ is a complete Riemannian metric on $\left(M \setminus \{o\} \right) \# T^n$, whose scalar curvature is nonnegative and strictly positive somewhere. This contradicts with Theorem 3 in \cite{Chodosh-Li-Annals} and Theorem 1.1 in \cite{Wang-Zhang-22}.
\qed

\section{Rigidity of metrics with conical singularity and zero scalar curvature}\label{sect-rigidity}

In this section, we prove the rigidity result in Theorem \ref{thm-main-1}.

\begin{proof}[Proof of rigidity result in \thmref{thm-main-1}]. 
By the non-existence result in Theorem \ref{thm-main-1}, a metric $g$ with nonnegative scalar curvature and isolated conical singularity must be scalar flat. 
Now, we first prove that $g$ must be Ricci flat by contradiction. Assume that $g$ is Riemannian metric on $M \# T^n$ with a conical singularity at $o \in M$, which is scalar flat, i.e. $\Sc_g =0$, but not Ricci flat, i.e. $\Ric(g)$ is non-zero at some point. Take a smooth cut-off function $\eta$ on $\left( M \# T^n \right) \setminus \{o\}$ , which is equal to zero in a neighborhood of $o$ and is equal to 1 at a point where $\Ric_g \neq 0$. Let $g_t = g - t \eta \Ric_g$ for small $t \geq 0$. Clearly, $g_t = g$ on a neighborhood of $o$ by the choice of the cut-off function $\eta$, and so $g_t$ are metrics with an isolated conical singularity at $o$. We consider the conformal Laplacian of the metric $g_t$:
\begin{equation}
L_{t} := -\Delta_{g_t} + \tfrac{n-2}{4(n-1)} \Sc_{g_t}.
\end{equation}
Recall that the scalar curvature of $g$ near the conical singularity is given by
\begin{equation}
\Sc_g = \frac{1}{r^2}\left( \Sc_{g^N} - (n-1)(n-2) + O(r^\alpha) \right).
\end{equation}
Thus, $\Sc_g \equiv 0$ implies that the scalar curvature of the cross section: $\Sc_{g^N} = (n-1)(n-2) >0$. Then by the spectral property in Theorem 1.1 and Remark 1.3 in \cite{DW-JGA-2018}, the conformal Laplacian $L_t$ have discrete eigenvalues
\begin{equation}
\lambda_1(L_t) < \lambda_2(L_t) \leq \lambda_3(L_t) \leq \cdots.
\end{equation}
Note that the first eigenvalue is simple, i.e. the space of the corresponding eigenfunctions is $1$-dimensional, and one can take an eigenfunction $u_{1, t} >0$ on $\left( M \# T^n \right) \setminus \{o\}$ (see Theorem 6.3 in \cite{DW-JGA-2018}). This implies that $\lambda_1(L_t)$ is a smooth function of $t$. Similar as Proposition 9.1 in \cite{DW-JGA-2018}, we have
\begin{equation}\label{eqn-variation-of-eigenvalue}
\left. \frac{d}{dt}\right|_{t=0} \lambda_{1}(L_t) = C \int_{M\# T^n} \langle - \Ric_g, - \eta \Ric_g \rangle d\vol_{g} >0,
\end{equation}
since $\Ric_g \not\equiv 0$. In the derivation of this variation formula, we have use that at $t=0$ the metric $g$ has scalar curvature identically zero, so the conformal Laplacian $L_0$ has zero as the first eigenvalue and constant functions are corresponding eigenfunctions. The variation formula \eqref{eqn-variation-of-eigenvalue} implies that for a sufficiently small $t_0 >0$, $\lambda_1(L_{t_0})>0$. We take a corresponding eigenfunction $u_{1, t_0} >0$. Moreover, by Theorem 8.3 in \cite{DW-2022}, near the conical singularity $o$, the eigenfunction $u_{1, t_0}$ has the partial asymptotic expansion as:
\begin{equation}\label{eqn-eigenfunction-asymptotic}
    u_{1, t_0} = a + o(r^{\nu^\prime}), \ \ \text{as}  \ \ r \to 0,
\end{equation}
for a constant $a$ and $\nu^\prime > 0$. Recall that $r=0$ corresponds the conically singular point $o$. Because $u_{1, t_0} >0$ on $\left( M \# T^n \right) \setminus \{o\}$, the constant $a \geq 0$. Take a sufficiently small $\delta >0$, and let $\widetilde{u_{1, t_0}}:= u_{1, t_0} + \delta$, and
\begin{equation}
\widetilde{g_{t_0}} := \left( \widetilde{u_{1, t_0}}\right)^{\tfrac{4}{n-2}} g_{t_0}.
\end{equation}
Then $\widetilde{g_{t_0}}$ is a Riemannian metric on $M \# T^n$ with a single conical singularity at $o$, and its scalar curvature is given by
\begin{eqnarray*}
\Sc_{\widetilde{g_{t_0}}}
& = & \tfrac{4(n-1)}{n-2} \left( \widetilde{u_{1, t_0}} \right)^{-\tfrac{n+2}{n-2}} \left( L_{t_0} u_{1, t_0} + \delta\tfrac{n-2}{4(n-1)} \Sc_{g_{t_0}} \right) \\
& = & \tfrac{4(n-1)}{n-2} \left( \widetilde{u_{1, t_0}} \right)^{-\tfrac{n+2}{n-2}} \left( \lambda_{1, t_0} u_{1, t_0} + \delta\tfrac{n-2}{4(n-1)} \Sc_{g_{t_0}} \right).
\end{eqnarray*}
Note that in a sufficiently small neighborhood of $o$ the metric $g_{t_0} = g$ and so $\Sc_{g_{t_0}} = \Sc_g = 0$, and outside of this neighborhood the function $u_{1, t_0}$ has a positive lower bound. Thus, we can take $\delta$ sufficiently small so that $\Sc_{\widetilde{g_{t_0}}} \geq 0$ and strictly positive somewhere. This contradicts with the non-existence result in \thmref{thm-main-1}, and so $\Ric_g \equiv 0$.

Then by Theorem A in \cite{BKMR-AIF-21}, the Riemannian manifold with a single conical singularity, $(M \# T^n, g)$, at $o \in M$, endowed with the distance $d_g$ and the measure $v_g$ induced by $g$, satisfies the ${\rm RCD}(0, n)$ condition. Note that because the dimension $n \geq 3$ and the only singular point is $o$, the stratum $\Sigma^{n-2} = \emptyset$, and so the singular Ricci curvature bounded below by $0$ is equivalent to the Ricci curvature on the regular set bounded below by $0$.  Moreover, the metric space $(M \# T^n, d_g)$ clearly is a connected and locally simply connected, and so it has the simply connected universal cover. As a result, the revised fundamental group  of $(M \# T^n, d_g)$ in \cite{Mondino-Wei-19} is the same as the usual fundamental group, and so it contains $n$ independent generators of infinite order, coming from the fundamental group of $T^n$. Thus, Corollary 1.4 in \cite{Mondino-Wei-19} implies that $(M \# T^n, d_g, v_g)$ is isomorphic as a metric measure space to a flat torus $T^n$. Actually, from their proof in \cite{Mondino-Wei-19}, one can see that $(M\# T^n, g)$ is isometric to a flat torus as Riemannian manifold. 
\end{proof}


\appendix

\section{Schoen-Yau-Lohkamp compactification with conical singularity}\label{sect-compactification}
In this appendix, we make some remarks about Schoen-Yau-Lohkamp compactification for asymptotically flat (AF for short) manifolds with finitely many isolated conical singularities.

The Schoen-Yau-Lohkamp compactification consists of two steps. Let $(M^n, g, o)$ be an AF manifold with a conical singularity at $o$. In the first step, one shows that if $\Sc_g \geq 0$ but the ADM mass $m(g) <0$, then there exists an asymptotically flat metric $\tilde{g}$ on $M$ with a conical singularity at $o$ such that $\Sc_{\tilde{g}} \geq 0$, and $\tilde{g} = \tilde{u}^{\frac{4}{n-2}}g_{\mathbb{R}^n}$ outside a large compact set with $\tilde{u} = 1 + \frac{\tilde{m}}{\rho^{n-2}} + O(\rho^{1-n}) (\tilde{m} <0)$ and $\Sc_{\tilde{g}} =0$. Here $\tilde{u}$ is the product of $1+\frac{m(g)}{4(n-1)\rho^{n-2}}$ and the solution to the conformal Laplacian equation (with truncated scalar curvature as the potential) on $M$. Here the conical singularity makes some difference in analysis from the smooth case, and so we provide some derivation for this part in Proposition \ref{prop-compactification-step1} below. 

In the second step, Lohkamp's main observation is that one can truncate the conformal factor $\tilde{u}$ on the AF end to further change the metric $\tilde{g}$ so that the new metric is a constant multiple of $g_{\mathbb{R}^n}$ outside a large compact set in the AF end, while keeping positive scalar curvature. Note that $\Sc_{\tilde{g}} =0$ implies $\Delta_{g_{\mathbb{R}^n}} \tilde{u} =0$ outside a large compact set. This fact and $\tilde{m}<0$ are crucial to guarantee the positivity of the scalar curvature of the new metric. By cutting $M$ off outside a larger compact set $K$ such that $\partial K = \partial [0, 1]^n$ and then gluing the opposite faces of the cube $[0, 1]^n$ together, we obtain  $T^n \# M^n_1$ equipped with a metric that has positive scalar curvature and a conical singularity. This contradicts Theorems \ref{thm-main-1} and \ref{thm-mfld-with-bdry}, thereby establishing the positive mass theorem with isolated conical singularity. This step only works on an AF end, so there is no difference between the conically singular case and the smooth case, for details, see Proposition 6.1 in \cite{Lohkamp1999}.

We will use the following Dirichlet Green's function on AF manifold with boundary in the proof of Lemma \ref{prop-compactification-solve-equation} below.
\begin{prop}[Green's funcction on AF manifold with boundary]\label{prop-Green-function}
Let $(M^n, g)$ be an $n$-dimensional AF smooth manifold with smooth boundary. 
Then there exists a positive Green's function $G(x, y)$ which is smooth on $\left( \mathring{M} \times \mathring{M} \right) \setminus D$, where $D=\{(x, x) \mid x \in \mathring{M}\}$, satisfying properties:
\begin{equation}\label{eqn: laplace G composition}
\int_{M} G(x, y) \Delta f(y) d\vol_g(y) = - f(x), 
\end{equation}
for smooth functions $f$ with $f|_{\partial M} =0$ and $f(x) \to 0$ as $x \to \infty$,
\begin{equation}
G(x, y) = G(y, x),
\end{equation}
and $G(x, y) = 0$ for $y \in \partial M$ and $ x \in M \setminus \partial M$.

Moreover, let $x_0$ be a point in the AF end, $G(x, x_0)$ has the following expansion, when $\|x-x_0\|$ is large, 
\begin{equation}
G(x, x_0) = \frac{C}{\|x - x_0\|^{n-2}} + O\left( \frac{1}{\|x - x_0\|^{\min\{n-1, k-2+\tau\}}}\right)
\end{equation}
for some constant $C >0$, where $\tau$ is the decay order of the AF metric $g$, $\|x - x_0\|$ is the Euclidean distance between $x$ and $x_0$. Its gradient also has the asymptotic estimate
\begin{equation}
\|\nabla G(x, x_0)\| = O\left( \frac{1}{\|x-x_0\|^{n-1}} \right).
\end{equation}
\end{prop}
\begin{proof}
Recall that the existence of the Green's function on complete manifold without boundary is proved in Theorem in \cite{Li-book} by taking limit for Green's functions on a compact exhaustion $\{\Omega_i\}^\infty_{i=1}$ of the manifold. We note that this proof can be adapted to the complete manifolds with boundary. We can take a compact exhaustion such that $\partial \Omega_i \supset \partial M$ for all $i$. For AF manifolds, we can take $f(x):= \|x-x_0\|^{2-n-\epsilon}$ on the AF end as a positive superharmonic function in Theorem 17.3 in \cite{Li-book}, where $x_0$ is a fixed point on the AF end, $\|x - x_0\|$ is the Euclidean distance between $x$ and $x_0$, and $\epsilon$ is a small positive number. For the detail of the proof, we refer to the proof of Theorem 17.3 in \cite{Li-book}. Moreover, the asymptotic behavior of the Green's function at AF infinity is proved in Proposition 11 in \cite{DaiSun-PMT}.
\end{proof}

Now as analysis preparation for proving Proposition \ref{prop-compactification-step1}, we solve an equation in the following lemma.
\begin{lem}\label{prop-compactification-solve-equation}
Let $(M^n, g, o)$ be an $n$-dimensional asymptotically flat manifold with an isolated conical singularity at $o$. There exists a constant $\epsilon_0 = \epsilon(g) $, such that if $f$ is a smooth function with compact support in $M_{\infty}$ and $ \| f_{-} \|_{L^{\frac{n}{2}}} < \epsilon_0$, then the equation
\begin{equation}\label{EQu}
\begin{cases}
\Delta_g u - fu =0 \ \ \text{on}  \ \ M, \\
u \to 1 \ \ \text{as} \ \ x \to \infty
\end{cases}
\end{equation}  
has a unique positive solution. Moreover, near the asymptotically flat infinity, $u$ satisfying:
\begin{equation}\label{eqn-asymptotic-AF}
u = 1 -\frac{A}{\rho^{n-2}} + \omega, \ \ \text{as} \ \ \rho \to \infty,
\end{equation}
where $A = C \int_{M} fu d\vol_g$ for some constant $ C>0$, $\rho$ is the Euclidean distance to a base point, and $\omega$ satisfies
\begin{equation}\label{eqn: control of higher order term of u}
|\omega| \leq \frac{C}{1+\rho^{n-1}}, \ \ |\nabla \omega| \leq \frac{C}{1+\rho^{n}}, \ \ |\nabla^2 \omega | \leq \frac{C}{1+\rho^{n+1}}.
\end{equation}
Near the conically singular point $o$, the solution $u$ has the asymptotic as
\begin{equation}\label{eqn-asymptotic-cone}
u  = B + o(r^\delta), \ \ \text{as} \ \ r \to 0,
\end{equation}
for a constant $B\geq 1$ and $\delta >0$, where $r$ is the radial variable on the model cone and $r=0$ corresponds to the conical singularity.
\end{lem}
\begin{proof}
The existence of the solution and the asymptotic behavior in (\ref{eqn-asymptotic-AF}) and (\ref{eqn-asymptotic-cone}) have been proved in Lemma 5.2 in \cite{DSW-PMT-nonspin}; and for the proof details, we refer to the proof that lemma. Here, we mainly prove that the constant $A$ in (\ref{eqn-asymptotic-AF}) can be expressed as $A = C \int_{M} fu d\vol_g$ for some constant $C>0$.
 
First of all, by setting $v=1-u$, we convert the equation \eqref{EQu} into
 \begin{equation}\label{FNEQ}
    \begin{cases}
      \Delta_{g} v-fv=-f \ \ \text{on}\; M, \\
      v\to 0 \ \ \text{as}\; x \to \infty.
    \end{cases}
   \end{equation}
Then an approximation argument by using a compact exhaustion shows the existence of the solution to (\ref{FNEQ}), see the proof of Lemma 5.2 in \cite{DSW-PMT-nonspin} for the details.

In the following, we use the asymptotic estimate of the Green's function $G$ and $\nabla G$ in Proposition \ref{prop-Green-function} to obtain the asymptotic behavior of the solution $v$. 
 Let $0< \rho_1 < \rho_2$ be sufficiently large, and $\phi$ be a smooth cut-off function on  the AF end $M_\infty$, which is an AF manifold with boundary, such that
 \begin{equation*}
 \phi =
     \begin{cases}
     1, & \text{on} \, M_\infty \setminus B_{\rho_2}(0), \\
     0, & \text{on} \, M_\infty \cap B_{\rho_1}(0).
     \end{cases}
 \end{equation*}
Let $G$ be the Green's function on $M_{\infty}$ obtained in Proposition \ref{prop-Green-function}. Then for any $x\in M_\infty$, we have
\begin{eqnarray*}
  & & (\phi v)(x) \\
  &=&\int_{M_\infty} G(x, z) \Delta_{g}(\phi v)(z)d\operatorname{vol}_{g}(z) \\
  &=& \int_{M_\infty}\left[\phi(z) G(x,z) (fv-f)(z)+2G(x,z)\langle\nabla\phi,\nabla v\rangle(z)+G(x,z)(\Delta \phi)(z)v(z)\right]d\operatorname{vol}_{g}(z)\\
 &=& \int_{M_\infty}\left[\phi(z) G(x,z) (fv-f)(z)-G(x,z)\langle\nabla(1-\phi),\nabla v\rangle\right(z)]d\operatorname{vol}_{g}(z)\\
 &&-\int_{M_\infty}\langle\nabla G(x,z),\nabla \phi(z)\rangle v(z)d\operatorname{vol}_{g}(z)\\
  &=& \int_{M_\infty}\left[\phi(z) G(x,z) (fv-f)(z)+G(x,z)(1-\phi)\Delta v\right]d\operatorname{vol}_{g}(z)\\
 && + \int_{M_\infty}(1-\phi(z))\langle\nabla G(x,z),\nabla v(z)\rangle d\operatorname{vol}_{g}(z)
 -\int_{M_\infty}\langle\nabla G(x,z),\nabla \phi(z)\rangle v(z)d\operatorname{vol}_{g}(z)\\
  &=& \int_{M_\infty}G(x,z)(fv-f)(z)d\operatorname{vol}_{g}(z)+\int_{M_\infty}(1-\phi(z))\langle\nabla G(x,z),\nabla v(z)\rangle d\operatorname{vol}_{g}(z)\\
  &&-\int_{M_\infty}\langle\nabla G(x,z),\nabla \phi(z)\rangle v(z)d\operatorname{vol}_{g}(z).
\end{eqnarray*}
Therefore
\begin{eqnarray*}
  \lim_{|x|\to\infty}|x|^{n-2}v((x)) &=& \int_{M_\infty}\lim_{|x|\to\infty}|x|^{n-2}G(x,z)(fv-f)(z)d\operatorname{vol}_{g}(z)\\
  &&-\int_{M_\infty}\lim_{|x|\to\infty}|x|^{n-2}\langle\nabla G(x,z),\nabla \phi(z)\rangle v(z)d\operatorname{vol}_{g}(z)\\
  &&+\int_{M_\infty}\lim_{|x|\to\infty}|x|^{n-2}(1-\phi(z))\langle\nabla G(x,z),\nabla v(z)\rangle d\operatorname{vol}_{g}(z).\\
  &=&C\int_{M}(fv-f)d\operatorname{vol}_{g}
\end{eqnarray*}
for $C>0$. Here we used ${\rm supp}(f) \subset M_\infty$ and the asymptotic estimates of the Green's function $G$ and its gradient $\nabla G$ in Proposition \ref{prop-Green-function}.

\end{proof}

Then we complete the first step of Schoen-Yau-Lohkamp compactification argument with a conical singularity in the following proposition.
\begin{prop}\label{prop-compactification-step1}
Let $(M^n, g, o)$ be an asymptotically flat manifold with a single conical singularity at $o$. Assume the scalar curvature $\Sc_g \geq 0$ and ADM mass $m(g) <0$. Then there exists an asymptotically flat metric $\tilde{g}$ on $M$ with a single conical singularity at $o$ such that $\tilde{g} = \tilde{u}^{\frac{4}{n-2}} g_{\R^n}$ outside a compact set, with $\tilde{u} = 1 + \frac{\tilde{m}}{\rho^{n-2}} + O(\rho^{1-n})$ for some $\tilde{m} <0$. Moreover, $\Sc_{\tilde{g}} \geq 0$ on $M$ and $\Sc_{\tilde{g}} =0$ outside a large compact set.
\end{prop}
\begin{proof}
    Following the proof of Proposition 4.11 in \cite{CLSZ} and Proposition 3.2 in \cite{Zhu-IRMN-23}, write the metric $g$ as
   \begin{equation}
     g = (1+\tfrac{m_{1}}{\rho^{n-2}})^{\tfrac{4}{n-2}}g_{\mathbb{R}^n}+\bar{g}
   \end{equation}
outside a large compact set $o \in K$ with $m_{1} :=\frac{1}{4(n-1)}m(g) <0$, and 
   \begin{eqnarray}\label{eqnmasszero}
     \lim_{\rho\to\infty}\int_{S_{\rho}}\sum_{i,j}(\partial_{i}\bar{g}_{ij}-\partial_{j}\bar{g}_{ii})\tfrac{x^{j}}{\rho}dS_{\rho}&=&0.
   \end{eqnarray}
     Let $\phi(\rho)$ be a cut-off function $\phi(\rho)=1$ for $\rho\leq 2$, and $\phi(\rho)=0$ for $\rho\geq 3$, and $0\leq\phi(\rho)\leq 1$.
Define the metric 
   \begin{equation}
   g^{\sigma}= 
           \begin{cases}
              g, & \text{on} \ \ K, \\
             (1+\tfrac{m_{1}}{\rho^{n-2}})^{\frac{4}{n-2}}g_{\mathbb{R}^n}+\phi(\tfrac{\rho}{\sigma})\bar{g}, & \text{on} \ \ M \setminus K.
            \end{cases} 
   \end{equation}
Then $\Sc_{g^{\sigma}}$ is a smooth function with support in $K \cup B_{3\sigma}$. More precisely,
\begin{eqnarray*}
  \Sc_{g^{\sigma}} &=&\begin{cases}
                       \Sc_{g}, & \mbox{if } \rho\leq 2\sigma \ \ \text{and on } K, \\
                       O(\sigma^{-\tau-2}), & \mbox{if }2\sigma\leq \rho\leq 3\sigma, \\
                       0, & \mbox{otherwise}.
                     \end{cases}
\end{eqnarray*}  
Consider the equation:
\begin{equation}\label{zeroscalarcurvature}
\begin{cases}
   \Delta_{g^{\sigma}}u-\tfrac{n-2}{4(n-1)}\varphi \Sc_{g^{\sigma}}u = 0, \\
   u \to 1\quad \text{as} \quad x\to\infty, 
\end{cases}
\end{equation}
where $\varphi$ is cut-off function such that $\varphi(\rho)=1$ for $2\sigma\leq \rho\leq 3\sigma$, and $\varphi(\rho)=0$ for $0\leq \rho\leq \sigma$, $4\sigma\leq \rho<\infty$, and on $K$.
Choose $\sigma$ sufficiency large to make $(\int_{M}|(\varphi \Sc_{g^{\sigma}})_{-}|^{n/2}d\operatorname{vol}_{g^{\sigma}})^{2/n}=O(\sigma^{-\tau})$ small. Then by Lemma \ref{prop-compactification-solve-equation}, the equation (\ref{zeroscalarcurvature}) has a positive solution $u$ with asymptotic behavior at asymptotically flat infinity as:
$$u=1-\tfrac{A_{\sigma}}{\rho^{n-2}}+\omega$$
and
$$A_{\sigma}=C\int_{M}\varphi \Sc_{g^{\sigma}}ud\operatorname{vol}_{g^{\sigma}}.$$
Let $\tilde{g}=u^{\frac{4}{n-2}}g^{\sigma}$. By the asymptotic behavior of $u$ as in (\ref{eqn-asymptotic-AF}) and (\ref{eqn-asymptotic-cone}), $\tilde{g}$ is still an asymptotically flat metric with a conical singularity at $o$. Then 
\begin{eqnarray*}
  \Sc_{\tilde{g}} &=& \frac{4(n-1)}{n-2}u^{-\frac{n+2}{n-2}}(-\Delta_{g^{\sigma}} u+\tfrac{n-2}{4(n-1)} \Sc_{g^{\sigma}}u) \\
   &=& \frac{4(n-1)}{n-2}u^{-\tfrac{n+2}{n-2}}(-\Delta_{g^{\sigma}} u+\tfrac{n-2}{4(n-1)}\varphi \Sc_{g^{\sigma}}u+\tfrac{n-2}{4(n-1)}(1-\varphi) \Sc_{g^{\sigma}}u)\\
   &=&u^{-\frac{4}{n-2}}(1-\varphi) \Sc_{g^{\sigma}}\geq 0.
\end{eqnarray*}
Moreover,
$$m(\tilde{g})=-4(n-1)A_{\sigma}+m(g^{\sigma})=-4(n-1)A_{\sigma}+m$$ 
and thus 
\begin{eqnarray}
  |m(\tilde{g})-m(g)|=4(n-1)|A_{\sigma}|.
\end{eqnarray}

Then one can show that $|A_{\sigma}|$ can be made arbitrarily small for sufficiently large $\sigma$. Note that for estimating $A_\sigma$ one only needs to work on AF end, so it is the same as smooth case, for the detail see e.g., Proof of step 1 in \cite{DaiSun-PMT}. As a result, $|m(\tilde{g})-m(g)|$ can be arbitrarily small for sufficiently large $\sigma$, so $m(g)<0$ implies $m(\tilde{g})<0$ by taking $\sigma$ sufficiently large. Moreover, note that outside a compact set $\tilde{g} = u^{\frac{4}{n-2}}(1+\frac{m_1}{\rho^{n-2}})^{\frac{4}{n-2}} g_{\mathbb{R}^n} =: \tilde{u}^{\frac{4}{n-2}}g_{\mathbb{R}^n}$, and as $\rho \to \infty$,
\begin{equation*}
\tilde{u} = (1 + \tfrac{m_1}{\rho^{n-2}}) (1 - \tfrac{A_\sigma}{\rho^{n-2}} + O(\rho^{1-n})) = 1 + \tfrac{m_1 - A_{\sigma}}{\rho^{n-2}} + O(\rho^{1-n}).
\end{equation*}
Here $\tilde{m}:= m_1 - A_{\sigma} <0$ for sufficiently large $\sigma$, since $|A_\sigma|$ can be arbitrary small for sufficiently large $\sigma$.
  \end{proof}

\bibliographystyle{plain}
\bibliography{DSW_PSC.bib}

\end{document}